\documentclass[12pt,letterpaper]{amsart}

\usepackage[utf8]{inputenc}
\usepackage[T1]{fontenc}
\usepackage{mathpazo}
\usepackage{graphicx}
\usepackage{booktabs}
\usepackage{listings}
\usepackage{amsmath, amssymb, amsthm}
\usepackage{enumerate}
\usepackage[table]{xcolor}
\usepackage{pgffor}
\usepackage[makeroom]{cancel}
\usepackage[margin=1in]{geometry}
\usepackage{enumitem}
\allowdisplaybreaks

\newtheorem{thm}{Theorem}[section]
\newtheorem{lem}[thm]{Lemma}
\newtheorem{quest}[thm]{Question}
\newtheorem{prop}[thm]{Proposition}

\newcommand{\R}{\mathbb R}

\newcommand{\Z}{\mathbb Z}
\newcommand{\G}{\mathrm{Mod}^{\pm}(\Sigma_2)}

\newcommand{\GL}{\mathrm{GL}}
\newcommand{\PGL}{\mathrm{PGL}}

\newcommand{\Id}{\mathrm{Id}}

\newcommand{\Mod}[2]{\mathrm{Mod}_{#1}\hspace{-0.3em}\left(#2\right)}
\newcommand{\EMod}[2]{\mathrm{Mod}^\pm_{#1}\hspace{-0.3em}\left(#2\right)}
\newcommand{\MS}[1]{\mathrm{Mod}(\Sigma_{0,#1})}
\newcommand{\EMS}[1]{\mathrm{Mod}^\pm(\Sigma_{0,#1})}

\newcommand{\s}[2][]{\sigma_{#2}^{#1}}

\begin{document}
\title{Generating Extended Mapping Class Groups with Two Periodic Elements}
\author{Reid Harris}
\date{}
\begin{abstract}The extended mapping class group of a surface $\Sigma$ is defined to be the group of isotopy classes of (not necessarily orientation-preserving) homeomorphisms of $\Sigma$. We are able to show that the extended mapping class group of an $n$-punctured sphere is generated by two elements of finite order exactly when $n\not=4$. We use this result to prove that the extended mapping class group of a genus 2 surface is generated by two elements of finite order.
\end{abstract}

\maketitle

\section{Introduction}

Let $\Sigma_{g,n}$ be an orientable, genus $g$ surface with $n$ punctures and let $\Sigma_g = \Sigma_{g,0}$. We let $\Mod{}{\Sigma_{g,n}}$ denote the mapping class group of $\Sigma_{g,n}$, i.e. isotopy classes of orientation-preserving homeomorphisms $\Sigma_{g,n}\to\Sigma_{g,n}$, and let $\EMod{}{\Sigma_{g,n}}$ be the corresponding extended mapping class group, i.e. isotopy classes of orientation-preserving or reversing homemorphisms $\Sigma_{g,n}\to\Sigma_{g,n}$. Our concern in this paper will mainly be on the groups $\EMod{}{\Sigma_2}$ and $\EMod{}{\Sigma_{0,n}}$. We consider the following question:

\begin{quest}
Find minimal generating sets $S$ of $\EMod{}{\Sigma_{g,n}}$ such that each element of $S$ is of finite order.
\end{quest}

\subsection{Previous Work}  The problem of finding generating sets, all of whose elements satisfy a given property (e.g. finite order), is classical and has been extensively studied. In 1938, Dehn \cite{Dehn1938}, proved that $\Mod{}{\Sigma_{g,0}}$ was generated by $2g(g-1)$ Dehn twists for $g\ge 3$. Later, in 1964, Lickorish, \cite{Lick1964}, improved this to $g\ge 1$ and reduced the number of Dehn twists needed to $3g-1$. This was reduced further still to $2g+1$ in 1977 by Humphries, \cite{Hump1977}, using a subset of Lickorish's generating set. Johnson, \cite{John1983}, showed in 1983 that Humphries' generators also generate $\Mod{}{\Sigma_{g,1}}$ for $g\ge 1$. Wajnryb showed in 1996 that $\Mod{}{\Sigma_{g,n}}$ can be generated by two elements, however, these elements are not Dehn twists.

In regards to torsion generating sets, Maclachlan \cite{Mac1971} showed that $\Mod{}{\Sigma_{g}}$ is generated by a finite set of torsion elements, concluding that moduli space is simply-connected. Luo \cite{Luo2000} showed that $\Mod{}{\Sigma_{g,n}}$ is generated by torsion elements, giving specific bounds for the order of generators given $(g,n)$. In particular, he shows that $\Mod{}{\Sigma_{g, n}}$ is generated by a involutions for $g\ge 2$. Brendle and Farb \cite{BF2004} show that $\Mod{}{\Sigma_{g,n}}$, for $g\ge 1$, is generated by three elements of finite order and for $g\ge 3, n=0$ and $g\ge 4, n=1$, $\Mod{}{\Sigma_{g,n}}$ is generated by six involutions. Kassobov \cite{Kass2003} shows that $\Mod{}{\Sigma_{g,n}}$ can be generated by
\begin{itemize}[noitemsep]
    \item[] 4 involutions if $g>7$ or $g=7$ and $n$ is even,
    \item[] 5 involutions if $g>5$ or $g=5$ and $n$ is even,
    \item[] 6 involutions if $g>3$ or $g=3$ and $n$ is even,
    \item[] 9 involutions if $g=3$ and $n$ is odd.
\end{itemize}
Korkmaz shows in \cite{Kork2005} that $\Mod{}{\Sigma_g}$ is generated by two elements of finite order and later showed in \cite{Kork2019} that $\Mod{}{\Sigma_g}$ is generated by three involutions for $g\ge 8$ and four involutions for $g\ge 3$. Yildiz \cite{Yil2020} shows that $\Mod{}{\Sigma_g}$ is generated by two elements of order $g$ for $g\ge 6$.

However, the corresponding question about $\EMod{}{\Sigma_{g,n}}$ remains largely unanswered. Du showed in \cite{Du2017}, \cite{Du2019} that $\EMod{}{\Sigma_1}\cong\GL_2(\Z)$ cannot be generated by two elements of finite order and, for $g>2$, the group $\EMod{}{\Sigma_g}$ is generated by two elements of finite order. Later, Altun\"oz et. al. in \cite{APY2023} showed that $\EMod{}{\Sigma_g}$ is generated by three involutions for $g\ge 5$ and, moreover, $\EMod{}{\Sigma_{g,n}}$ can be generated by three involutions for $g= 10$, $n\ge 6$ or $g\ge 11$, $n\ge 15$. In \cite{Mon2024}, Monden shows that, for $g\ge 3$ and $n\ge 0$, the groups $\Mod{}{\Sigma_{g,n}}$ and $\EMod{}{\Sigma_{g,n}}$ are generated by two elements.

The question of whether $\EMod{}{\Sigma_2}$ can be generated by such elements remained open. In this paper, we answer in the affirmative. In the course of the proof, we show that
\begin{thm}
\label{mainthm}
The group $\EMod{}{\Sigma_{g,n}}$ can be generated by finite order elements for $g=0, n\not=4$ and $g=2, n=0$. Moreover, $\EMod{}{\Sigma_{0,4}}$ cannot be generated by finite order elements.
\end{thm}

\subsection{Acknowledgements}

I would like to express my sincerest gratitute to Dr. Du Xiaoming for suggesting this problem to me, for his conversations at South China University of Technology, and for his advice and comments on an earlier draft of the paper. I would also like to express my gratitute to Dr. Hou Yong for giving me the opportunity to work with him and his group at CUHK(SZ).

\section{Preliminaries}

\subsection{Spherical Braid Group}

Given any surface $\Sigma$, the classical braid group can be generalized to the {\it braid group on $\Sigma$}, denoted $B_n(\Sigma) := \pi_1(\mathrm{Conf}_n(\Sigma))$, where $\mathrm{Conf}_n(\Sigma)$ is the space of unordered configurations of $n$ distinct points on $\Sigma$. In particular, we will be interested in the {\it spherical braid groups} $B_n(S^2)$. We have a surjective homomorphism $B_n\to B_n(S^2)$ with kernel generated by the central element $R_n:= \sigma_1\dots\sigma_{n-1}\sigma_{n-1}\dots\sigma_1$. Then $B_n(S^2)$ has the presentation given by generators $\tilde\sigma_1,\dots, \tilde\sigma_{n-1}$ and relations
\begin{itemize}[noitemsep]
\item $\tilde\sigma_i\tilde\sigma_j = \tilde\sigma_j\tilde\sigma_i$ for $|i-j| > 2$
\item $\tilde\sigma_i\tilde\sigma_j\tilde\sigma_i=\tilde\sigma_j\tilde\sigma_i\tilde\sigma_j$ for $|i-j|=1$
\item $R_n=1$.
\end{itemize}

We turn our attention to the relationship between $B_n(S^2)$ and $\Mod{}{\Sigma_{0,n}}$. We have the exact sequence
\begin{align}
\label{braidtomcgses}
0\to \langle\beta\rangle\to B_n(S^2)\xrightarrow{\psi}\Mod{}{\Sigma_{0,n}}\to 0
\end{align}
where $\beta=(\tilde\sigma_1\dots\tilde\sigma_{n-1})^n$ and $\langle\beta\rangle\cong\Z/2\Z$ (see \cite{FM}, Section 9.1.4 and 9.2).

Here, we let $\sigma_i = \psi(\tilde\sigma_i)$ for $1\le i\le n-1$. Since we are interested in elements of finite order, we record the following result:
\newpage
\begin{prop}
    \label{periodicspherebraid}
    The elements of $\Mod{}{\Sigma_{0,n}}$ of finite order are conjugate to a power of one of the following:
    \begin{figure}[!h]
    \centering
    \begin{tabular}{|c|c|c|}
    \hline
       Element & Factoring & Order\\
    \hline
        $\alpha_0$ & $\sigma_1\dots\sigma_{n-1}$ & $n$\\
    \hline
        $\alpha_1$ & $\sigma_1\dots\sigma_{n-2}$ & $n-1$\\
    \hline
        $\alpha_2$ & $\sigma_1\dots\sigma_{n-3}\sigma_{n-2}^2$ & $n-2$\\
    \hline
    \end{tabular}
    \end{figure}
\end{prop}

\begin{proof}
    Let $\tilde\sigma_i$ refer to the standard generators of $B_n(S^2)$. Let $f\in\Mod{}{\Sigma_{0,n}}$ such that $f^k=1$. There exists a lift $\tilde f\in B_n(S^2)$. Thus, $\tilde f^k$ is a power of $\beta\in B_n(S^2)$, from (\ref{braidtomcgses}), which has finite order and so $\tilde f$ is also periodic. From \cite{Mur1982}, $\tilde f$ must be conjugate to a power of one of
    \begin{itemize}[noitemsep]
    	\item $\tilde\sigma_1\dots\tilde\sigma_{n-1}$,
	\item $\tilde\sigma_1\dots\tilde\sigma_{n-2}\tilde\sigma_{n-1}^2$, or
	\item $\tilde\sigma_1\dots\tilde\sigma_{n-3}\tilde\sigma_{n-2}^2$.
    \end{itemize}
    Note that $(\sigma_1\dots\sigma_{n-2}\sigma_{n-1}^2)^{-1} = \sigma_{n-2}\dots\sigma_1$ is conjugate to $\sigma_1\dots\sigma_{n-2}$ in $\Mod{}{\Sigma_{0,n}}$. To see this, suppose $\Sigma_{0,n}$ is the unit sphere in $\R^3$ and arrange the marked points $p_1,\dots,p_n$ in order and uniformly along the equator of the sphere. Define $\phi:\Sigma_{0,n}\to\Sigma_{0,n}$ by rotating $\pi$ radians along the axis through $p_n$ and the center of $\Sigma_{0,n}$. Then, $$[\phi]\cdot \sigma_i\cdot [\phi]^{-1} = \sigma_{n-1-i}$$ for all $1\le i\le n-2$. Hence, $f$ is conjugate to a power of one of the elements in the table.
\end{proof}

We will also make use of the following relations, which hold in $\Mod{}{\Sigma_{0,n,0}}$:
\begin{align}
\label{rotationactstwist}
    \alpha_0\sigma_i\alpha_0^{-1} &= \sigma_{i+1}\text{ for }1\le i < n-1\\
    \alpha_1\sigma_i\alpha_1^{-1} &= \sigma_{i+1}\text{ for }1\le i < n-2\\
    \alpha_2\sigma_i\alpha_2^{-1} &= \sigma_{i+1}\text{ for }1\le i < n-3
\end{align}
In particular, $\Mod{}{\Sigma_{0,n,0}}$ is generated by $\sigma_1$ and $\alpha_0$.

\subsubsection{Birman-Hilden}

We introduce the Birman-Hilden exact sequence for $\Sigma_2$. For details, see \cite{BH1973} and \cite{FM}.

\begin{thm}[Birman-Hilden]
\label{bhthm}
Let $\iota\in\Mod{}{\Sigma_2}$ denote the mapping class of an involution on $\Sigma_2$ with 6 fixed points. There is an exact sequence 
\begin{align}
        \label{BHseq}
    0\to \langle\iota\rangle\to \Mod{}{\Sigma_2}\to \Mod{}{\Sigma_{0,6}}\to 0.
    \end{align}
\end{thm}

The following result will be useful in Section \ref{MainThm} to prove part of the main theorem. It extends the Birman-Hilden exact sequence to the extended mapping class group.

\begin{prop}
\label{BHThm}
Let $\iota\in\Mod{}{\Sigma_2}$ denote the mapping class of an involution on $\Sigma_2$ with 6 fixed points. There is an exact sequence $$0\to \langle\iota\rangle\to \EMod{}{\Sigma_2}\xrightarrow{\Psi} \EMod{}{\Sigma_{0,6}}\to 0.$$
\end{prop}

\begin{proof}
    Let $\phi\in\EMod{}{\Sigma_2}$ be orientation-reversing. Since there exists an orientation-reversing homeomorphism $T:\Sigma_2\to\Sigma_2$ which is fiber-preserving, we may pick a representative $f:\Sigma_2\to\Sigma_2$ of $\phi$ which is fiber-preserving: there is a representative $g$ of $[T]\phi$ which is fiber preserving by \cite{BH1973} and so we may take $f = T^{-1}\circ g$. Letting $\pi:\Sigma_2\to\Sigma_{0,6}$ denote the branched covering map, we define $\bar f:\Sigma_{0,6}\to\Sigma_{0,6}$ by $\bar f= \pi\circ f\circ\pi^{-1}$.

    Suppose $f$ and $f'$ are both representatives of $\phi$, that is, $f$ and $f'$ are isotopic. Then $T\circ f$ and $T\circ f'$ are orientation-preserving, isotopic and fiber-preserving. By Theorem \ref{bhthm}, these maps are isotopic through fiber-preserving homemorphisms, say $H:\Sigma_2\times[0,1]\to\Sigma_2$ is such an isotopy. Hence, $H' = T^{-1}\circ H$ is a fiber-preserving isotopy between $f$ and $f'$. This isotopy then descends to an isotopy between $\bar f$ and $\bar f'$. Thus, we have a well-defined map $\Psi:\EMod{}{\Sigma_2}\to\EMod{}{\Sigma_{0,6}}$ given by $[f]\mapsto [\bar f]$. Since $\Psi|_{\Mod{}{\Sigma_2}}$ is exactly the Birman-Hilden homomorphism from (\ref{BHseq}) and the kernel of this map must lie in $\Mod{}{\Sigma_2}$, we see that $\ker(\Psi) = \langle\iota\rangle$.
\end{proof}

\section{Periodic Elements in $\EMS{n}$}

Let $n\ge 1$. For our standard model of $\Sigma_{0,n}$, we take the unit sphere embedded in $\R^3$ along with marked points $p_k$, $k=0,\dots,n-1$, given by $$p_k = \left(\cos\frac{2\pi k}{n},\; \sin\frac{2\pi k}{n},\; 0\right).$$ Let $T:\Sigma_{0,n}\to \Sigma_{0,n}$ denote the map given by $T(x,y,z) = (x,y,-z).$ We also let $T$ denote the isotopy class of this homeomorphism in $\EMod{}{\Sigma_{0,n}}$. Let $\sigma_i$, for $1\le i\le n-1$, denote the mapping class of the right Dehn twist about the arc connecting $p_i$ to $p_{i+1}$ along the equator. Note that $T\sigma_i = \sigma_i^{-1}T$ for each $1\le i\le n-1$.

We have the following presentation for $\EMS{n}$: generators are $\sigma_1,\dots,\sigma_{n-1}$, and $T$ with relations
\begin{itemize}[noitemsep]
\item $T^2 = (T\sigma_i)^2=1$, for $1\le i\le n-1$,
\item $\sigma_i\sigma_j=\sigma_j\sigma_i$, for $|i-j|\ge 2$,
\item $\sigma_i\sigma_j\sigma_i=\sigma_j\sigma_i\sigma_j$, for $|i-j|=1$,
\item $(\sigma_1\dots\sigma_{n-1})^n=1$,
\item $\sigma_1\dots\sigma_{n-1}\sigma_{n-1}\dots\sigma_1=1$
\end{itemize}

This is the presentation obtained from the isomorphism $\EMS{n}\cong\MS{n}\rtimes\Z/2\Z$ where the non-identity element $T$ of $\Z/2\Z$ acts on $\MS{n}$ by $\sigma_i\mapsto\sigma_i^{-1}$.

Recall that the orientation-preserving mapping classes of finite order are given by Proposition \ref{periodicspherebraid}. Using the presentation above, we have that
\begin{align*}
T\alpha_0T &= \sigma_{1}^{-1}\dots\sigma_{n-1}^{-1}\\
	&=  (\sigma_1\dots\sigma_{n-1}\sigma_{n-1}\dots\sigma_1)\cdot\sigma_{1}^{-1}\dots\sigma_{n-1}^{-1}\\
	&= \sigma_1\dots\sigma_{n-1}\\
	&= \alpha_0.
\end{align*}
Thus, $T\alpha_0$ is periodic with order $n$ if $n$ is even and order $2n$ if $n$ is odd. We also easily see that $$(T\s{1}\s{3}\dots\s{2k-1})^2 = 1,$$ for each $k=0,\dots,\lfloor n/2\rfloor$. Lastly,
\begin{align*}
    (T\sigma_{n-1}^{-1}) \alpha_2 (T\sigma_{n-1}^{-1}) &= T\sigma_{n-1}^{-1}\alpha_0\sigma_{n-1}^{-1}\sigma_{n-2}T\sigma_{n-1}^{-1}\\
    &=\sigma_{n-1}\alpha_0\sigma_{n-1}\sigma_{n-2}^{-1}\sigma_{n-1}^{-1}\\
    &=\alpha_0\sigma_{n-2}\sigma_{n-1}\sigma_{n-2}^{-1}\sigma_{n-1}^{-1}\\
    &=\alpha_0\sigma_{n-1}^{-1}\sigma_{n-2}\\
    &=\alpha_2.
\end{align*}
Thus, $T\sigma_{n-1}^{-1}$ and $\alpha_2$ commute and $T\s[-1]{n-1}\alpha_2$ has order $n-2$ if $n$ is even or $2(n-2)$ if $n$ is odd.

For general $n$, these do not exhaust all possibilities of orientation-reversion periodic elements, even up to conjugacy. For example, when $n=9$, there exists an orientation-reversing mapping class of order 6, acting by the permutation $(1\;2\;3\;4\;5\;6)(7\;8\;9)$ on the marked points, which is not covered by any of the above examples or their powers. However, it would be interesting to find a classification of all finite-order elements of $\EMod{}{\Sigma_{0,n}}$ in terms of the generators $\sigma_i$.

\section{Proof of Main Theorem}
This section is divided into 3 subsections, each dealing with a proof of particular case of Theorem \ref{mainthm}.

\subsection{$\EMS{4}$ cannot be generated by two periodic elements}

\begin{thm}
The group $\EMS{4}$ cannot be generated by two elements of finite order.
\end{thm}

\begin{proof}
    Consider the short exact sequence 
    \begin{align}
    \label{glgexact}
    0\to \langle -\Id\rangle\to\GL_2(\Z)\xrightarrow{q}\PGL_2(\Z)\to 0.
    \end{align}
    If $\bar A\in\PGL_2(\Z)$ has $\bar{A}^k=\Id\in\PGL_2(\Z)$, then for any representative $A$ of $\bar A$, $A^k = \pm\Id$ so $A$ is periodic.
    Suppose that $\PGL_2(\Z)$ is generated by two elements $\bar A, \bar B$ of finite order. Then, if $A,B$ are representatives of $\bar A, \bar B$, then $A$ and $B$ generate a subgroup $H$ of $\GL_2(\Z)$. For any $g\in \GL_2(\Z)$, the only representatives of $q(g)$ are $g$ and $-g$, so either $g\in H$ or $-g\in H$. Hence, the index $[\GL_2(\Z):H] \le 2$. Thus, $\GL_2(\Z)/H$ is abelian and $[\GL_2(\Z),\GL_2(\Z)]\le H$. Note that $-\Id = [x,y]$, where
    $$
   	 x = \left(\begin{array}{cc}0&1\\1&0\end{array}\right)\text{ and } y=\left(\begin{array}{cc}-1&0\\0&1\end{array}\right).
    $$ Thus, $-\Id\in H$. But then $H=-H$ and so $[\GL_2(\Z):H] = 1$ which contradicts the result from \cite{Du2019}. Therefore, $\PGL_2(\Z)$ cannot be generated by two elements of finite order. Since we have a surjection $\EMS{4}\to\PGL_2(\Z)$, see Section 2.2.5 of \cite{FM}, the group $\EMS{4}$ cannot be generated by two finite order elements.
\end{proof}

Note that $\EMod{}{\Sigma_{0,4}}$ can be generated by the three periodic elements $T$, $T\sigma_1$, and $\alpha_0$.

\subsection{Periodic generation of $\EMS{n}$, for $n\not=4$}
\label{SectionMainTheoremProof}

We begin with a simple observation:
\begin{prop}
If $n$ is odd, then $\EMS{n}$ is generated by $T\sigma_1$ and $T\alpha_0$.
\end{prop}

\begin{proof}
Let $H:=\langle T\sigma_1,T\alpha_0\rangle$. We have that $$(T\alpha_0)^n = T^n\alpha_0^n = T.$$ Therefore, $T\in H$ and so $\sigma_1,\alpha_0\in H$. Since $\sigma_1$ and $\alpha_0$ generate $\MS{n}$, we have $\MS{n}\le H$, but since $T\in H\setminus\MS{n}$, we must have that $H=\EMS{n}$.
\end{proof}

This proposition shows that for odd $n$, the theorem is immediate since $T\sigma_1$ has order $2$ and $T\alpha_0$ has order $2n$. We now turn to the more difficult case.

\begin{thm}\label{MainThm}
For all even $n\ge 6$, $\EMS{n}$ is generated by $a = \s{n-3}T\alpha_0\s[-1]{n-3}$ and $b = T\s[-1]{n-1}\alpha_2$.
\end{thm}

To prove this, we proceed in a sequence of steps. Let $H = \langle a, b\rangle$. We will make use of the following relations. For $k\not= n-6, n-4, n-2$,
\begin{align*}
    a^2\s{k}a^{-2} &= \s{n-3}\alpha_0^2\s[-1]{n-3}\cdot \s{k}\cdot \s{n-3}\alpha_0^{-2}\s[-1]{n-3}\\
    &= \s{n-3}\alpha_0^2\cdot \s{k}\cdot \alpha_0^{-2}\s[-1]{n-3}\\
    &= \s{n-3}\s{k+2}\s[-1]{n-3}\\
    &= \s{k+2}.
\end{align*}

\begin{lem}
We have $$y:=\prod_{\substack{k=1\\ k\text{ odd}}}^{n-1}\s{k}=\s{1}\s{3}\dots\s{n-1}\in H.$$
\end{lem}

\begin{proof}
We first compute the following:
\begin{align*}
x_0	&= b^{-2}ab\\
	&=\left(\alpha_2^{-2}\right)\cdot \left(\s{n-3}T\alpha_0\s[-1]{n-3}\right)\cdot \left(T\s[-1]{n-1}\alpha_2\right)\\
	&=\left(\s[-1]{n-2}\s{n-1}\alpha_0^{-1}\right)\left(\s[-1]{n-2}\s{n-1}\alpha_0^{-1}\right)\cdot \s{n-3}\boxed{T\alpha_0\s[-1]{n-3}T}\s[-1]{n-1}\alpha_0\s[-1]{n-1}\s{n-2}\\
	&=\s[-1]{n-2}\s{n-1}\alpha_0^{-1}\s[-1]{n-2}\s{n-1}\alpha_0^{-1}\cdot \s{n-3}\boxed{\alpha_0\s{n-3}}\s[-1]{n-1}\alpha_0\s[-1]{n-1}\s{n-2}\\
	&=\s[-1]{n-2}\cancel{\s{n-1}}\s[-1]{n-3}\cancel{\s{n-2}} \s{n-5}\s{n-4}\cancel{\s[-1]{n-2}}\cancel{\s[-1]{n-1}}\s{n-2}\\
	&=\s[-1]{n-2}\s[-1]{n-3} \s{n-5}\s{n-4}\s{n-2}\\
 \\
x_1 	&= x_0ax_0^{-1}\\
	&=\left(\s[-1]{n-2}\s[-1]{n-3} \s{n-5}\s{n-4}\s{n-2}\right)\cdot \s{n-3}T\alpha_0\s[-1]{n-3} \cdot\left(\s[-1]{n-2}\s[-1]{n-4}\s[-1]{n-5}\s{n-3}\s{n-2}\right)\\
	&=\s[-1]{n-2}\s[-1]{n-3} \s{n-5}\s{n-4}\s{n-2}\cdot \s{n-3}T\alpha_0\s[-1]{n-3} \cdot\s[-1]{n-2}\s[-1]{n-4}\s[-1]{n-5}\s{n-3}\s{n-2}\\
	&=\s[-1]{n-2}\s[-1]{n-3} \s{n-5}\s{n-4}\s{n-2}\s{n-3}\alpha_0\s{n-3}\s{n-2}\s{n-4}\s{n-5}\s[-1]{n-3}\s[-1]{n-2}T\\
	&=\s[-1]{n-2}\s[-1]{n-3} \s{n-5}\s{n-4}\s{n-2}\s{n-3}\s{n-2}\s{n-1}\s{n-3}\s{n-4}\s[-1]{n-2}\s[-1]{n-1}T\alpha_0\\
 \\
x_2 	&= x_1a^{-1}\\
	&= \s[-1]{n-2}\s[-1]{n-3} \s{n-5}\s{n-4}\s{n-2}\s{n-3}\s{n-2}\s{n-1}\s{n-3}\s{n-4}\s[-1]{n-2}\s[-1]{n-1}\boxed{T\alpha_0\cdot \s{n-3}T\alpha_0^{-1}}\s[-1]{n-3}\\
	&= \s[-1]{n-2}\s[-1]{n-3} \s{n-5}\s{n-4}\s{n-2}\s{n-3}\s{n-2}\s{n-1}\s{n-3}\s{n-4}\s[-1]{n-2}\s[-1]{n-1}\boxed{\s[-1]{n-2}}\s[-1]{n-3}\\
	&= \s[-1]{n-2}\s[-1]{n-3} \s{n-5}\s{n-4}\s{n-2}\s{n-3}\s{n-2}\s{n-1}\s{n-3}\s{n-4}\boxed{\s[-1]{n-2}\s[-1]{n-1}\s[-1]{n-2}}\s[-1]{n-3}\\
	&= \s[-1]{n-2}\s[-1]{n-3} \s{n-5}\s{n-4}\s{n-2}\s{n-3}\s{n-2}\cancel{\s{n-1}}\s{n-3}\s{n-4}\boxed{\cancel{\s[-1]{n-1}}\s[-1]{n-2}\s[-1]{n-1}}\s[-1]{n-3}\\
	&= \s[-1]{n-2}\s[-1]{n-3} \s{n-5}\s{n-4}\s{n-2}\boxed{\s{n-3}\s{n-2}\s{n-3}}\s{n-4}\s[-1]{n-2}\s[-1]{n-1}\s[-1]{n-3}\\
	&= \s[-1]{n-2}\s[-1]{n-3} \s{n-5}\s{n-4}\s{n-2}\boxed{\s{n-2}\s{n-3}\cancel{\s{n-2}}}\s{n-4}\cancel{\s[-1]{n-2}}\s[-1]{n-1}\s[-1]{n-3}\\
	&=\boxed{ \s[-1]{n-2}\s[-1]{n-3} \s{n-5}}\boxed{\s{n-4}\s{n-2}\s{n-2}}\s{n-3}\s{n-4}\s[-1]{n-1}\s[-1]{n-3}\\
	&=\boxed{\s{n-5} \s[-1]{n-2}\s[-1]{n-3}}\boxed{\s{n-2}\s{n-2}\s{n-4}}\s{n-3}\s{n-4}\s[-1]{n-1}\s[-1]{n-3}\\
	&=\s{n-5} \s[-1]{n-2}\s[-1]{n-3}\s{n-2}\s{n-2}\boxed{\s{n-4}\s{n-3}\s{n-4}}\s[-1]{n-1}\s[-1]{n-3}\\
	&=\s{n-5} \s[-1]{n-2}\s[-1]{n-3}\s{n-2}\s{n-2}\boxed{\s{n-3}\s{n-4}\cancel{\s{n-3}}}\s[-1]{n-1}\cancel{\s[-1]{n-3}}\\
	&=\s{n-5} \s[-1]{n-2}\s[-1]{n-3}\s{n-2}\s{n-2}\s{n-3}\s{n-4}\s[-1]{n-1}\\
 \\
x_3 &= x_2b^{-1}\\
	&= \s{n-5} \s[-1]{n-2}\s[-1]{n-3}\s{n-2}\s{n-2}\s{n-3}\s{n-4}\s[-1]{n-1}\cdot  \s[-1]{n-2}\s{n-1}\boxed{\alpha_0^{-1}\s{n-1}}T\\
	&= \s{n-5} \s[-1]{n-2}\s[-1]{n-3}\s{n-2}\s{n-2}\s{n-3}\s{n-4}\s[-1]{n-1}\s[-1]{n-2}\s{n-1}\boxed{\s{n-2}\alpha_0^{-1}}T\\
	&= \s{n-5} \s[-1]{n-2}\s[-1]{n-3}\s{n-2}\s{n-2}\s{n-3}\s{n-4}\s[-1]{n-1}\boxed{\s[-1]{n-2}\s{n-1}\s{n-2}}\alpha_0^{-1}T\\
	&= \s{n-5} \s[-1]{n-2}\s[-1]{n-3}\s{n-2}\s{n-2}\s{n-3}\s{n-4}\cancel{\s[-1]{n-1}}\boxed{\cancel{\s{n-1}}\s{n-2}\s[-1]{n-1}}\alpha_0^{-1}T\\
	&= \s{n-5} \s[-1]{n-2}\s[-1]{n-3}\s{n-2}\s{n-2}\s{n-3}\boxed{\s{n-4}\s{n-2}}\s[-1]{n-1}\alpha_0^{-1}T\\
	&= \s{n-5} \s[-1]{n-2}\s[-1]{n-3}\s{n-2}\s{n-2}\s{n-3}\boxed{\s{n-2}\s{n-4}}\s[-1]{n-1}\alpha_0^{-1}T\\
	&= \s{n-5} \s[-1]{n-2}\s[-1]{n-3}\s{n-2}\boxed{\s{n-2}\s{n-3}\s{n-2}}\s{n-4}\s[-1]{n-1}\alpha_0^{-1}T\\
	&= \s{n-5} \s[-1]{n-2}\s[-1]{n-3}\s{n-2}\boxed{\s{n-3}\s{n-2}\s{n-3}}\s{n-4}\s[-1]{n-1}\alpha_0^{-1}T\\
	&= \s{n-5} \s[-1]{n-2}\s[-1]{n-3}\boxed{\s{n-2}\s{n-3}\s{n-2}}\s{n-3}\s{n-4}\s[-1]{n-1}\alpha_0^{-1}T\\
	&= \s{n-5} \cancel{\s[-1]{n-2}}\cancel{\s[-1]{n-3}}\boxed{\cancel{\s{n-3}}\cancel{\s{n-2}}\s{n-3}}\s{n-3}\s{n-4}\s[-1]{n-1}\alpha_0^{-1}T\\
	&= \s{n-5} \s{n-3}\s{n-3}\s{n-4}\s[-1]{n-1}\alpha_0^{-1}T\\
 \\
x_4 &= x_3a\\
	&=\s{n-5} \s{n-3}\s{n-3}\s{n-4}\s[-1]{n-1}\boxed{\alpha_0^{-1}T\cdot \s{n-3}T\alpha_0}\s[-1]{n-3}\\
	&=\s{n-5} \s{n-3}\cancel{\s{n-3}}\cancel{\s{n-4}}\s[-1]{n-1}\boxed{\cancel{\s[-1]{n-4}}}\cancel{\s[-1]{n-3}}\\
	&=\s{n-5} \s{n-3}\s[-1]{n-1}\\
\end{align*}

\noindent Define $\gamma_k := \s{k}\s{k+2}\s[-1]{k+4}$ where subscripts are taken modulo $n$. Also,
\begin{align*}
    a^{2k}\gamma_1a^{-2k} &= a^{2k}\s{1}\s{3}\s[-1]{5}a^{-2k}\\
        &= a^{2k}\s{2k+1}\s{2k+3}\s[-1]{2k+5}a^{-2k}\\
        &= \gamma_{2k+1}
\end{align*} for all odd $k$. The above computations show that $\gamma_{n-5} \in H$. Hence, $\gamma_k\in H$ for all odd $k$. Thus,
\begin{align*}
    y   &= \gamma_1\gamma_3\dots\gamma_{n-1}\\
        &= \s{1}\s{3}\dots\s{n-3}\s{n-1}\\
        &\in H.
\end{align*}

One can see this by noting that each pair of the $\sigma_i$'s which appear in $y$ commute and hence, the right-hand side can be obtained by adding exponents for each $\sigma_i$ which appears.
\end{proof}

\begin{lem}
We have $$z := \s{n-2}\prod_{\substack{k=1\\ k\text{ odd}}}^{n-5} \s{k} = \s{1}\s{3}\dots\s{n-5}\s{n-2}\in H.$$
\end{lem}

\begin{proof}
We start with
\begin{align*}
ab 	&= \s{n-3}T\alpha_0\s[-1]{n-3}\cdot T\s[-1]{n-1}\alpha_0\s[-1]{n-1}\s{n-2}\\
	&= \s{n-3}\alpha_0\s{n-3}\s[-1]{n-1}\alpha_0\s[-1]{n-1}\s{n-2}\\
	&= \left(\alpha_0\s[-1]{n-1}\right)^2\s{n-5}\s{n-4}\s{n-2}
\end{align*}
Let $\Delta_k := \s{k}\s{k+1}\s{k+3}$ for $1\le k\le n-5$. Then,
$$
\left(\alpha_0\s[-1]{n-1}\right)^2\Delta_{k} = \Delta_{k+2}\left(\alpha_0\s[-1]{n-1}\right)^2
$$
for $1\le k\le n-7$ and 
\begin{align*}
ab  &= \left(\alpha_0\s[-1]{n-1}\right)^2\Delta_{n-5}\\
    &= \alpha_0\s[-1]{n-1}\alpha_0\s[-1]{n-1}\s{n-5}\s{n-4}\s{n-2}\\
    &=\alpha_0^2\s[-1]{n-2}\s[-1]{n-1}\s{n-5}\s{n-4}\s{n-2}\\
    &=\alpha_0^2\s{n-5}\s{n-4}\s[-1]{n-2}\s[-1]{n-1}\s{n-2}\\
    &=\s{n-3}\s{n-2}\alpha_0^2\s{n-1}\s[-1]{n-2}\s[-1]{n-1}\\
    &=\s{n-3}\s{n-2}\s{1}\left(\alpha_0\s[-1]{n-1}\right)^2.
\end{align*}
\begin{align*}
\Delta_1\Delta_3\Delta_5\dots\Delta_{n-5} &= \s{1}\s{2}\s{4}\cdot\s{3}\s{4}\s{6}\cdot\s{5}\s{6}\s{8}\dots\s{n-7}\s{n-6}\s{n-4}\cdot\s{n-5}\s{n-4}\s{n-2}\\
    &= \s{1}\s{2}\s{3}\cdot\s{4}\s{3}\s{5}\cdot\s{6}\s{5}\s{7}\dots\s{n-6}\s{n-7}\s{n-5}\cdot\s{n-4}\s{n-5}\s{n-2}\\
    &=\s{1}\s{2}\dots\s{n-4}\cdot\s{3}\s{5}\dots\s{n-5}\s{n-2}\\
    &=\alpha_0\s[-1]{n-1}\s[-1]{n-2}\s[-1]{n-3}\cdot\s{3}\s{5}\dots\s{n-5}\s{n-2}\\
    &=\alpha_0\s[-1]{n-1}\s[-1]{n-2}\s[-1]{n-3}\cdot\s[-1]{1}z.
\end{align*}
Therefore, 
\begin{align*}
(ab)^{\frac{n}{2}-1} 
    &= \bigg[\left(\alpha_0\s[-1]{n-1}\right)^2\Delta_{n-5}\bigg]\cdot\left(\alpha_0\s[-1]{n-1}\right)^2\Delta_{n-5}\dots\left(\alpha_0\s[-1]{n-1}\right)^2\Delta_{n-5}\\
    &= \bigg[\s{n-3}\s{n-2}\s{1}\left(\alpha_0\s[-1]{n-1}\right)^2\bigg]\cdot\left(\alpha_0\s[-1]{n-1}\right)^2\Delta_{n-5}\dots\left(\alpha_0\s[-1]{n-1}\right)^2\Delta_{n-5}\\
    &= \s{n-3}\s{n-2}\s{1}\left(\alpha_0\s[-1]{n-1}\right)^{n-2}\Delta_{1}\Delta_{3}\Delta_{5}\dots\Delta_{n-7}\Delta_{n-5}\\
    &= \s{n-3}\s{n-2}\s{1}\bigg[\left(\alpha_0\s[-1]{n-1}\right)^{n-2} \cdot\alpha_0\s[-1]{n-1}\bigg]\s[-1]{n-2}\s[-1]{n-3}\cdot\s[-1]{1}z\\
    &= \s{n-3}\s{n-2}\s{1}\s[-1]{n-2}\s[-1]{n-3}\cdot\s[-1]{1}z\\
    &= z,
\end{align*}
where we use the fact that $\alpha_0\s[-1]{n-1} = \alpha_1$ has order $n-1$.
\end{proof}

\begin{proof}[Proof of Theorem \ref{MainThm}]
We have
\begin{align*}
w  :=& z^{-1}y\cdot \gamma_{n-3}^{-1}\\
    =& \s[-1]{n-2}\s{n-3}\s{n-1}\cdot \s[-1]{n-3}\s[-1]{n-1}\s{1}\\
    =& \s[-1]{n-2}\s{1}\\
    \in& H.
\end{align*}

Since $$a^{-1}b = \s{n-3}\s{n-4}\s[-1]{n-2}\s[-1]{n-1}\s{n-2},$$ we have that $$c:=a^{-1}b\cdot w\cdot b^{-1}a = \s[-1]{n-1}\s{1}.$$ Thus, $T\alpha_0\in H$ and, conjugating $\sigma_{n-3}$ by $T\alpha_0$ gives $\sigma_i\in H$ for all $1\le i\le n-1$.
\end{proof}

\subsection{Periodic generation of $\G$}

\begin{thm}
The group $\EMod{}{\Sigma_2}$ is generated by two elements of finite order.
\end{thm}
\begin{proof}
We have the exact sequence from Theorem \ref{BHThm}: \begin{align}\label{bhexact} 0\to \langle \iota\rangle \to \G \xrightarrow{q} \EMS{6} \to 0,\end{align} where $\iota$ is the mapping class of a hyperelliptic involution, so that $\langle \iota\rangle\cong \Z/2\Z$. Let $a, b$ be as in the previous theorem and let $\tilde a, \tilde b$ be preimages to $\G$. We claim that $\tilde a, \tilde b$ generate $\G$. Let $H = \langle \tilde a, \tilde b\rangle$ so that $q(H)=\EMS{6}$. For any $g\in\G$, we must have either $g\in H$ or $\iota g\in H$ since these are the only two preimages of $q(g)$. Hence, $[\G : H]\le 2$.

Suppose that $[\G:H]=2$. Then the quotient map $$\varphi:\G\to\G/H\cong \Z/2\Z$$ factors through the abelianization map $$\psi:\G\to(\Z/2\Z)^2,$$ say $\varphi = f\circ \psi$ for some $f:(\Z/2\Z)^2\to\Z/2\Z$. Let $\psi':\EMS{6}\to(\Z/2\Z)^2$ be the abelianization of $\EMS{6}$ given by $\psi'(\sigma_i)=(1,0)$, for $1\le i\le n-1$, and $\psi'(T) = (0,1)$. Since the hyperelliptic involution is a product of $10$ Dehn twists, its image in the abelianization is trivial (Section 5.1.3, \cite{FM}). Hence, $\psi = \psi'\circ q$. Since $$\psi(\tilde a) = \psi'(a) = (1,1)\text{ and }\psi(\tilde b) = \psi'(b) =  (0,1)$$ and $$f(1,1) = \varphi(\tilde a) = 0\text{ and }f(0,1) = \varphi(\tilde b) = 0,$$ we find that $f=0$ and $\varphi$ is not surjective. This gives a contradiction.

\end{proof}

\bibliographystyle{abbrv}

\end{document}